\documentclass[10pt]{amsart}
\usepackage{enumerate,amsmath,amssymb,latexsym,
amsfonts, amsthm, amscd, MnSymbol}

\theoremstyle{plain}

\newtheorem*{theorem}{Theorem}

\numberwithin{equation}{section}

\newcommand{\ra}{\rightarrow}

\begin{document}

\title {$\mathbf L^1$ completeness for Fourier series}

\date{}

\author[P.L. Robinson]{P.L. Robinson}

\address{Department of Mathematics \\ University of Florida \\ Gainesville FL 32611  USA }

\email[]{paulr@ufl.edu}

 \keywords{}

\begin{abstract}

We note that the Fubini theorem may be used to prove that an $L^1$ function is determined by its Fourier coefficients. 

\end{abstract}

\maketitle

\medbreak

One of the most fundamental results in the theory of Fourier series is the uniqueness theorem asserting that an $L^1$ function is essentially determined by its Fourier coefficients. Explicitly, the Fourier coefficients of $f \in L^1 [ - \pi, \pi ]$ are defined by 
$$a_n := \frac{1}{\pi} \int_{- \pi}^{\pi} f(t) \cos  n t \: {\rm d} t \; \; \; \; (n \geqslant 0)$$
$$b_n := \frac{1}{\pi} \int_{- \pi}^{\pi} f(t) \sin  n t \; {\rm d} t \; \; \; \; (n > 0)$$
\medbreak 
\noindent 
and the fundamental uniqueness theorem to which we refer states that if each of these Fourier coefficients is zero then the function $f$ itself is zero almost everywhere. 

\medbreak 

By now, this result has been established in many ways. One of the more elegant approaches involves the Ces\`aro means $\sigma_N = \frac{1}{N + 1} (s_0 + \cdots + s_N)$ of the partial sums 
$$s_N (t) = \frac{1}{2} a_0 + \sum_{n = 1}^{N} (a_n \cos n t + b_n \sin n t )$$
of the Fourier series of $f$. In 1904, Fej\'er proved that if $f$ is a continuous function then $\sigma_N \ra f$ pointwise and indeed uniformly; in 1905, Lebesgue proved that if $f$ is an $L^1$ function then $\sigma_N \ra f$ almost everywhere. The uniqueness theorem is an immediate consequence. 

\medbreak 

Our sole concern here is with another early approach to the proof of the uniqueness theorem due to Lebesgue. The details of this approach may be found in the classics [7] (pages 11-12) and [4] (pages 18-19); they may also be found in more recent texts such as [1] (pages 55-57), [3] (pages 40-41) and [2] (pages 226-228). This approach starts with the case in which $f$ is continuous and here employs an auxiliary sequence $(T_n)$ of trigonometric polynomials: [1] and [3] follow Zygmund in their choice of $T_n$; [2] follows Hardy and Rogosinski. All five texts are in substantial agreement on the continuation of the proof, in which $f$ is an $L^1$ function with vanishing Fourier coefficients and the continuous case is applied to a specific indefinite integral $F$ of $f$: integration by parts shows that the Fourier coefficients of the (absolutely) continuous function $F$ also vanish, whence $F$ is zero and therefore $f = 0$ almost everywhere. 

\medbreak 

In at least the three more recent texts, this reduction of the $L^1$ case to the continuous case rests clearly and firmly on the relationship between differentiation and integration in the Lebesgue theory. Times change. These relatively sophisticated aspects of the Lebesgue theory are now perhaps more likely to be covered in a second course, a first course being perhaps more likely to include the Fubini theorem pertaining to repeated integrals. For this reason, we suggest the following alternative proof of the uniqueness theorem (rather, of the reduction from the $L^1$ case to the continuous case). 

\begin{theorem} 
If the Fourier coefficients of $f \in L^1 [- \pi, \pi]$ all vanish then $f = 0$ almost everywhere. 
\end{theorem} 

\begin{proof} 
As was discussed above, we shall assume the truth of the corresponding statement for a continuous function. 

\medbreak 

Let $f \in L^1 [- \pi, \pi]$ but temporarily make no assumption regarding the vanishing of  its Fourier coefficients. Introduce the continuous function $F$ defined by 
$$F(t) = \int_0^t f(u) \; {\rm d} u - \frac{1}{2} a_0 t$$
(an indefinite integral of $f$ if $a_0$ vanishes) and note that $F( - \pi) = F( \pi )$ by direct calculation. We claim that the Fourier coefficients of $F$ satisfy 
$$A_n := \frac{1}{\pi} \int_{- \pi}^{\pi} F(t) \cos  n t \: {\rm d} t = -\frac{1}{n} b_n \; \; \; \; (n > 0)$$
and 
$$B_n := \frac{1}{\pi} \int_{- \pi}^{\pi} F(t) \sin  n t \; {\rm d} t = \frac{1}{n} a_n \; \; \; \; (n > 0).$$
\medbreak 
\noindent 
To justify this, we apply the Fubini theorem. The function $(t, u) \mapsto f(u) \cos n t$ is integrable over the square $[ - \pi, \pi ] \times [ - \pi, \pi ]$: inverting the order of integration over the indicated triangular subsets,  
\begin{eqnarray*}
\int_0^{\pi} \Big\{ \int_0^t f(u) \; {\rm d} u \Big\} \cos n t \; {\rm d} t & = & \int_0^{\pi} f(u) \; \Big\{ \int_u^{\pi} \cos n t \; {\rm d} t \Big\} \; {\rm d} u \\ & = & \int_0^{\pi} f(u) \Big[ \frac{\sin n t}{n} \Big]_u^{\pi} {\rm d} u
\end{eqnarray*}
so that 
$$\int_0^{\pi} \Big\{ \int_0^t f(u) \; {\rm d} u \Big\} \cos n t \; {\rm d} t = - \frac{1}{n} \int_0^{\pi} f(u) \sin n u \; {\rm d} u$$
and similarly 
$$\int_{-\pi}^0 \Big\{ \int_0^t f(u) \; {\rm d} u \Big\} \cos n t \; {\rm d} t = - \frac{1}{n} \int_{-\pi}^0 f(u) \sin n u \; {\rm d} u;$$
since 
$$\int_{-\pi}^{\pi}  \frac{1}{2} a_0 t \cos n t \; {\rm d} t = 0$$
it follows upon summation that 
$$A_n = - \frac{1}{n} b_n$$
as claimed. The calculation to establish that 
$$B_n = \frac{1}{n} a_n$$
is similar, being complicated only by the fact that the integral 
$$\int_{-\pi}^{\pi} \frac{1}{2} a_0 t \sin n t \; {\rm d} t = (-1)^{n - 1} \frac{\pi}{n} a_0$$
is precisely cancelled by a corresponding term in the integral 
$$\int_{-\pi}^{\pi} \Big\{ \int_0^t f(u) \; {\rm d} u \Big\} \sin n t \; {\rm d} t$$
on account of the definition
$$a_0 = \frac{1}{\pi} \int_{-\pi}^{\pi} f.$$
The function $F$ thus has Fourier coefficients $A_0$ (upon which we need not elaborate here), $A_n = - b_n/n$ and $B_n = a_n/n$ (for $n > 0$). 

\medbreak 

Now let us assume that the Fourier coefficients of $f$ vanish. It follows that the {\it continuous} function $F - \frac{1}{2} A_0$ has vanishing Fourier coefficients and is therefore identically zero; as $F(0) = 0$ it follows that $A_0 = 0$ and therefore that $F$ is identically zero. Thus $f$ has vanishing integral over each interval with $0$ as an endpoint and hence over each interval; so $f$ is zero almost everywhere. Of course, the last conclusion here may be drawn without reference to the relationship between integrals and derivatives. 

\end{proof} 

\medbreak 

Incidentally, if we prefer to follow the classics in proving the uniqueness theorem via partial integration, it is still possible to bypass the more sophisticated aspects of Lebesgue theory that relate integration and differentiation: we need only use a version of integration by parts that does not involve differentiation explicitly; such a version appears as Proposition 2 on page 103 of [6] and as the first theorem in section 26 on pages 54-55 of the classic [5]. On this point it is interesting to note that, in their proofs of the uniqueness theorem, neither Zygmund nor Hardy and Rogosinski seem to specify any  particular version of integration by parts. 

\medbreak 

Finally, it is perhaps worth mentioning that Hardy and Rogosinski [4] actually indicate several proofs of the uniqueness theorem: the original proof to which ours is an alternative, at page 18; at page 31, a variant based on termwise integration of Fourier series (our alternative line of argument being so arranged as to allow for its use here, too); at page 43, a variant making use of the fact that indefinite integrals are functions of bounded variation; at page 68, the elegant proof via Ces\`aro means. 

\bigbreak

\begin{center} 
{\small R}{\footnotesize EFERENCES}
\end{center} 
\medbreak 

[1] N. Bary, {\it A Treatise on Trigonometric Series}, Volume 1, Pergamon Press (1964). 

\medbreak 

[2] S.B. Chae, {\it Lebesgue Integration}, Second Edition, Springer-Verlag (1995). 

\medbreak 

[3] R.E. Edwards, {\it Fourier Series - A Modern Introduction}, Volume 1, Revised Edition, Springer-Verlag (1979). 

\medbreak 

[4] G.H. Hardy and W.W. Rogosinski, {\it Fourier Series}, Cambridge University Press (1944); Dover Publications (1999). 

\medbreak 

[5] F. Riesz and B. Sz.-Nagy, {\it Functional Analysis}, Frederick Ungar Publishing Company (1955); Dover Publications (1990). 

\medbreak 

[6] A.J. Weir, {\it Lebesgue Integration and Measure}, Cambridge University Press (1973). 

\medbreak 

[7] A. Zygmund, {\it Trigonometric Series}, Volume 1, Second Edition, Cambridge University Press (1959). 

\medbreak

\end{document}